\documentclass[a4paper,11pt,leqno]{amsart}
\usepackage{a4wide}
\usepackage{amsmath}
\usepackage[italian, english]{babel}
\usepackage{amsfonts}
\usepackage{amssymb}
\usepackage{dsfont}
\usepackage{mathrsfs}
\usepackage{graphicx}
\usepackage{multicol}
\usepackage{fancyhdr}
\usepackage{amsthm}
\usepackage{empheq}
\usepackage{cases}
\usepackage[all]{xy}
\usepackage{stmaryrd}
\usepackage[colorlinks, citecolor=blue, linkcolor=red]{hyperref}
\usepackage{mathrsfs}
\usepackage{tikz,tikz-cd}
\def\SS{{{\mathbb S}}}

\tikzset{
	subset/.style={
		draw=none,
		edge node={node [sloped, allow upside down, auto=false]{$\subset$}}},
	Subset/.style={
		draw=none,
		every to/.append style={
			edge node={node [sloped, allow upside down, auto=false]{$\subset$}}}
	}
}
\tikzset{
	labl/.style={anchor=south, rotate=90, inner sep=.50mm}
}

\newcommand{\erre}{\mathds{R}}

\newcommand{\ricc}{\operatorname{Ric}}
\newcommand{\weyl}{\operatorname{W}}

\newcommand{\bigslant}[2]{{\raisebox{.0em}{$#1$}\left/\raisebox{-.0em}{$#2$}\right.}}
\newcommand{\KN}{\mathbin{\bigcirc\mspace{-15mu}\wedge\mspace{3mu}}}

\newcommand{\longra}{\longrightarrow}

\newcommand{\pa}[1]{{\left(#1\right)}}                  % tra tonde
\newcommand{\sq}[1]{{\left[#1\right]}}                  % tra quadre
\newcommand{\abs}[1]{{\left|#1\right|}}                 % valore assoluto
\newcommand{\pair}[1]{\left\langle#1\right\rangle}      % pairing

\newcommand{\ol}[1]{\overline{#1}}

\renewcommand{\tilde}[1]{\widetilde{#1}}

\newcommand{\qd}[1]{\mathbin{{Q}_{#1}\mspace{-25mu}\tiny\raisebox{0.65ex}{$2$}\mspace{20mu}}}
\newcommand{\qt}[1]{\mathbin{{Q}_{#1}\mspace{-24.7mu}\tiny\raisebox{0.65ex}{$3$}\mspace{20mu}}}
\newcommand{\qq}[1]{\mathbin{{Q}_{#1}\mspace{-25.6mu}\tiny\raisebox{0.7ex}{$4$}\mspace{20mu}}}
\newcommand{\pmat}[1]{{\begin{pmatrix}#1\end{pmatrix}}} % matrice parentesi tonde
\newtheorem{theorem}{\textbf{Theorem}}[section]

\newtheorem{proposition}[theorem]{\textbf{Proposition}}

\theoremstyle{remark}
\newtheorem{rem}[theorem]{\textbf{Remark}}

\numberwithin{equation}{section}
\title[Riemannian four-manifolds and their twistor spaces]
{On Riemannian four-manifolds and their twistor spaces: a moving frame approach}



\begin{document}
\maketitle

%\date{\today}

\begin{center}
	\textsc{\textmd{Giovanni Catino\footnote{Politecnico di Milano, Italy.
				Email: giovanni.catino@polimi.it.}, Davide Dameno (*) \footnote{(*) Corresponding author. Universit\`{a} degli Studi di Milano, Italy.
				Email: davide.dameno@unimi.it.}, Paolo
			Mastrolia\footnote{Universit\`{a} degli Studi di Milano, Italy.
				Email: paolo.mastrolia@unimi.it.}. }}
\end{center}

\begin{abstract} In this paper we study the twistor space $Z$ of an oriented Riemannian four-manifold $M$ using the moving frame approach, focusing, in particular, on the Einstein, non-self-dual setting. We prove that any general first-order linear condition on the almost complex structures of $Z$ forces the underlying manifold $M$ to be self-dual, also recovering most of the known related rigidity results. Thus, we are naturally lead to consider first-order quadratic conditions, showing that the Atiyah-Hitchin-Singer almost Hermitian twistor space of an Einstein four-manifold bears a resemblance, in a suitable sense, to a nearly K\"ahler manifold.

\end{abstract}

\tableofcontents

\

\section{Introduction and main results}\label{secIntro}
Let $(M,g)$ be a Riemannian manifold of dimension $2m$, with metric
$g$. The \emph{twistor space} $Z$ associated to $M$ is defined as the set of all the
pairs $(p,J_p)$ such that $p\in M$ and $J_p$ is a linear endomorphism
of the tangent space $T_pM$ which satisfies the following conditions:
\begin{enumerate}
	\item for every $X,Y\in T_pM$, $g_p(J_p(X),J_p(Y))=g_p(X,Y)$;
	\item for every $X\in T_pM$, $J_p(J_p(X))=-X$.
\end{enumerate}
Such an endomorphism is called a $g$-\emph{orthogonal complex structure} on $T_pM$
\footnote{In this paper, we call \emph{complex structure} an endomorphism $J_V$ of a
	vector space $V$ such that $J_V^2=-\operatorname{Id}_V$, while we call
	\emph{almost complex structure} a $(1,1)$-tensor field $J$ on a differentiable manifold
	$M$ such that $J$ smoothly assigns, to every point $p$, a complex structure $J_p$
	on $T_pM$.}.

The twistor space $Z$ defines a fiber bundle over $M$ {\em via} the map that assigns to every
pair $(p,J_p)$ the point $p\in M$; hence, we can define $Z$ in an equivalent 
way as
\[
Z=\bigslant{O(M)}{U(m)},
\]
where $O(M)$ denotes the orthonormal frame bundle over $M$ and
the unitary group $U(m)$ is identified with a subgroup of $SO(2m)$ 
(see, for instance, \cite{debnan} and \cite{jenrig}). Moreover, 
if $M$ is oriented, $Z$ has two connected components $Z_{\pm}$ defined 
as quotients of the bundles $O(M)_{\pm}$ (which are subbundles
of $O(M)$) of
positively and negatively oriented orthonormal frames \emph{via} the
action of $U(m)$.
%given by
%\[
%\{A\in SO(2m):AJ_m=J_mA\},
%\]
%where $J_m$ is the representative matrix of the endomorphism of $T_pM$
%such that, for every orthonormal basis $\{e_1,\ldots,e_{2m}\}$ of $T_pM$,
%$J_m(e_{2j-1})=e_{2j}$, $j=1,\ldots,m$, and $J_m^2=-I_{2m}$.

%Throughout this paper,
%we consider only the case in which $M$ is oriented, in order
%to exploit the existence of two connected components
%$O(M)_+$ and $O(M)_-$ of $O(M)$ and,
%therefore, define the two connected components of $Z$ as
%\[
%Z_{\pm}=\bigslant{O(M)_{\pm}}{U(m)}=\bigslant{SO(M)}{U(m)},
%\]
%where $SO(M)$ is the orthonormal oriented frame bundle over $M$.
%Once we have chosen a connected component of $Z$, 
Twistor spaces can be regarded as Riemannian manifolds: indeed, it is possible
to define a natural family of Riemannian metrics $g_t$ on them,
where $t$ is a positive parameter, by taking the pullback of
a specific bilinear form defined on $O(M)$, as explained in
\cite{debnan} and in \cite{jenrig}.

These structures, introduced
by Penrose (\cite{penrose}), 
%as an attempt to define an innovative
%framework for Physics
have been the subject of many investigations
by the mathematical community, 
also in virtue of the numerous geometrical and algebraic tools involved in the definition of their properties. 
In 1978, Atiyah, Hitchin and
Singer (\cite{athisin}) exploited Penrose's twistor theory in the
Riemannian setting,
introducing the concept of twistor space associated to a Riemannian 
four-manifold in their study of self-dual Yang-Mills equations. 
Indeed, what can be observed is that there exist strong relations 
between the geometry of twistor spaces and the one of the underlying
Riemannian manifolds: many characterizations of
certain classes of Riemannian four-manifolds can be obtained by examining
the geometrical properties of their twistor spaces.

The particular interest for the four-dimensional geometry, beside the intrinsic importance due to the obvious relation with Relativity,  arises from the
unique structure of the Riemann curvature operator, which cannot
be realized in any other dimension.
Indeed, if $(M,g)$ is a Riemannian manifold of dimension
$m$, the Riemann curvature tensor $\operatorname{Riem}$ on $M$
admits the well known decomposition (see e.g. \cite{besse}, \cite{cmBook}
and Section \ref{secRiem} of this paper)
\[
\operatorname{Riem}=\weyl+\frac{1}{m-2}\operatorname{Ric}\KN g
-\frac{S}{2(m-1)(m-2)}g\KN g.
\]
%where $\operatorname{W}$, $\operatorname{Ric}$ and $S$ denote the
%\emph{Weyl tensor}, the \emph{Ricci tensor} and the \emph{scalar curvature}
%of $M$, respectively, and $\KN$ is the \emph{Kulkarni-Nomizu product}. Moreover,
Due to its symmetries, 
the Riemann curvature tensor defines a linear operator 
$\mathcal{R}$ from the bundle of two-forms $\Lambda^2$ to itself.
%\begin{align*}
%\mathcal{R}\colon\Lambda^2&\longrightarrow\Lambda^2\\
%\gamma&\longmapsto\mathcal{R}(\gamma)=\dfrac{1}{4}R_{ijkt}\gamma_{kt}\theta^i\wedge\theta^j,
%\end{align*}
%where $\{\theta^i\}_{i=1,\ldots,m}$ is a local orthonormal coframe on an
%open set $U\subset M$, with dual frame $\{e_i\}_{i=1,\ldots,2m}$,
%$\gamma_{kt}=\gamma(e_k,e_t)$ and $R_{ijkt}$ are the components
%of the Riemann tensor with respect to the coframe $\{\theta^i\}$ 
%in general, any $(0,4)$-tensor $P$ which satisfies the same symmetries
%as the Riemann curvature tensor induces a symmetric linear operator $\mathcal{P}$
%defined as above.
If $m=4$ and $M$ is oriented, $\Lambda^2$ splits, {\itshape via}
the Hodge $\star$ operator, into the direct sum of
two subbundles $\Lambda_+$ and $\Lambda_-$, which
are called the bundles of \emph{self-dual} and \emph{anti-self-dual} forms,
respectively. This implies that the Riemann curvature operator gives rise
to three linear maps $A$, $B$ and $C$, such that $A$ (resp, $C$) is a
symmetric endomorphism of $\Lambda_+$ (resp., $\Lambda_-$) and $B$ is
a linear map from $\Lambda_+$ to $\Lambda_-$ (see \cite{athisin}, \cite{besse}, \cite{singthor}
and \cite{friedr} for a complete dissertation. We would like to thank Prof.
Ilka Agricola for having pointed out to us this last reference);
therefore, $\mathcal{R}$ is represented by a block matrix
\[
\mathcal{R}=
\left(
\begin{array}{cc}
A & B^T\\
B & C
\end{array}
\right),
\]
where $A^T=A$, $C^T=C$. Moreover, $\operatorname{tr}A=\operatorname{tr}C=\frac{S}{4}$.
The splitting of $\Lambda^2$ also induces a decomposition of the Weyl tensor 
into a sum
\[
\weyl=\weyl^++\weyl^-,
\]
where $\weyl^+$ (resp., $\weyl^-$) is called the {\em self-dual} (resp., {\em anti-self-dual}) {\em part}
of $\weyl$. If $\weyl^+=0$ (resp., $\weyl^-=0$), we say that $M$ is an {\em anti-self-dual}
(resp., {\em self-dual}) {\em manifold}. 
A characterization of these \emph{half conformally flat} metrics in
terms of the decomposition of $\mathcal{R}$ is given by
Theorem \ref{cond4mani}.
%If we consider the symmetric linear operators induced
%by $W^+$ and $W^-$, we have that their representative matrices are
%$A-\frac{S}{12}I_3$ and $C-\frac{S}{12}I_3$, respectively, with respect to any
%positively oriented local orthonormal coframe; thus, $(M,g)$ is self-dual (resp.,
%anti-self-dual) if and only if $C=\frac{S}{12}I_3$ (resp., $A=\frac{S}{12}I_3$).
%Note that, if the coframe is negatively oriented, $A$ and $C$ need to be
%exchanged in the previous statements.

In the literature, many results about the relation between a Riemannian four-manifold
$M$ and its twistor space $Z$ were achieved starting from the hypothesis
that $M$ is half conformally flat: for instance, Atiyah, Hitchin and Singer, in \cite{athisin},
introduced an almost complex structure $J_+$ on $Z_{\pm}$ and showed that $M$ is
a self-dual (resp., anti-self-dual) manifold if and only if $J_+$ is integrable on $Z_-$ (resp., on $Z_+$), i.e. the associated
Nijenhuis tensor $N_{J_+}$ vanishes identically, while Eells and Salamon, in
\cite{elsal}, defined an
almost complex structure $J_-$ on $Z_{\pm}$ which is never integrable.
In 1985, Friedrich and Grunewald (\cite{frigru}) characterized 
Einstein twistor spaces $(Z,g_t)$ as the ones whose base manifold 
is Einstein, half conformally flat and with positive scalar curvature.
Some important characterization theorems for Einstein self-dual manifolds
were proved by Jensen and Rigoli (\cite{jenrig}), Friedrich and Kurke 
(\cite{frikur}) and by Davidov and Mu\v{s}karov
(see, for instance, 
\cite{mus}, \cite{davmusherm}, \cite{davmus1}, \cite{davmus2} 
and \cite{davmuskahl}), starting
from the classification of the almost Hermitian manifolds due to
Gray and Hervella (\cite{grayher}). In 1985, O' Brian and Rawnsley generalized the Atiyah-Hitchin-Singer twistor theory to higher dimensions, studying the problem of integrability of certain complex structures and proving a necessary and sufficient condition of integrability which involves the vanishing of the Bochner tensor for the underlying manifold (see \cite{obrian}). In the recent paper \cite{fuzhou}, the authors exploit the moving frame formalism to study, among other things, the so-called balanced and first Gauduchon metric conditions on the twistor spaces of a Riemannian four-manifold.

In this paper we start from the following questions:
\begin{enumerate}
    \item Is it possible to introduce a framework that could simplify the study of the Riemannian and Hermitian features of the twistor space
		associated to a Riemannian four-dimensional manifold?
	\item Given an \emph{Einstein four-manifold} $M$, is it possible to find new and interesting properties of its associated twistor space?

\end{enumerate}

Our approach to the aforementioned questions is inspired by the works of Jensen and Rigoli (\cite{jenrig}) and of Fu and Zhou (\cite{fuzhou}):
   all our computations of the main Riemannian and Hermitian features of the twistor spaces are based on the method of moving frames {\em \`a la} Cartan, which provides an effective answer to question $(1)$. As a consequence of our analysis we are able to easily recover and generalize some classical results. In particular, our first main result is the following
\begin{theorem}\label{TH:GeneralLinear}
	Let $(M,g)$ be an oriented Riemannian four-manifold and let
	$(Z_-,g_t,\bar{J})$ be
	its twistor space, with $\bar{J}=J_+$ or $\bar{J}=J_-$. Suppose that, for
	every $X,Y$ smooth vector fields on $Z_-$,
	\begin{align*}
	&a_1(\nabla_X \bar{J})Y+a_2(\nabla_Y\bar{J})X+a_3(\nabla_{\bar{J}X}\bar{J})Y
	+a_4(\nabla_{\bar{J}Y} \bar{J})X+
	a_5(\nabla_{\bar{J}X}\bar{J})\bar{J}Y+\\
	&+a_6(\nabla_{\bar{J}Y}\bar{J})\bar{J}X+a_7(\nabla_{X}\bar{J})\bar{J}Y+
	a_8(\nabla_{Y} \bar{J})\bar{J}X=0 \notag
	\end{align*}
	for some $a_i\in\mathbb{R}$, $i=1,\ldots,8$, such that $a_j\neq 0$ for some $j$. Then,
	$M$ is self-dual.
\end{theorem}
This theorem allows us to prove in an alternative way the integrability 
result on the Atiyah-Hitchin-Singer almost complex structure in 
\cite{athisin} and
the characterization results for Einstein, self-dual manifolds with positive scalar
curvature in \cite{mus}. Concerning question $(2)$, since one of our main 
goals is to study Einstein four-manifolds whose metrics are not necessarily 
self-dual, the previous theorem naturally lead  us to consider first-order 
quadratic conditions: more precisely, we are able to show a local (i.e., holding only
for \emph{orthonormal} frames/coframes), quadratic characterization of Einstein four-manifolds:

\begin{theorem}\label{TH:LocalQuadratic}
An oriented Riemannian four-manifold $(M,g)$ is Einstein if and only if, for every orthonormal frame in
$O(M)_-$ (equivalently, for every negatively oriented orthonormal coframe),
\begin{equation*}
\sum_{t=1}^6(J_{p,q}^t+J_{q,p}^t)(J_{p,p}^t-J_{q,q}^t)=0, \quad\forall p,q=1,\ldots,6,
\end{equation*}
where $J_{p,q}^t$ are the components of the covariant derivative of $J_+$ with respect to a local orthonormal coframe on $(Z_-,g_t,J_+)$.
\end{theorem}

Moreover, we can prove a quadratic necessary and sufficient condition for Einstein, non-self-dual manifolds:
\begin{theorem}\label{TH:LocalQuadratic2}
	Let $(M,g)$ be an oriented Riemannian Einstein four-manifold. Then, for every orthonormal frame in
$O(M)_-$,
 \begin{equation*}
   \pa{J^t_{q, p}+J^t_{p, q}}N_{pq}^t=0 \qquad \text{(no sum over $p$,$q$,$t$), }
 \end{equation*}
	where $J_{p,q}^t$ and $N_{pq}^t$ are the components of the covariant derivative of $J_+$ and of the Nijenhuis tensor, respectively, with
	respect to a local orthonormal coframe on $(Z_-,g_t,J_+)$. Conversely, if 
	the equation \eqref{nijeinst} holds on $O(M)_-$, then, for every point
	$p\in M$ such that $\abs{\weyl^-}\neq 0$ at $p$, the traceless Ricci
	tensor vanishes at $p$, i.e. $B=0$. In particular, if 
	$\abs{\weyl^-}\neq 0$ on $M$, $(M,g)$ is an Einstein manifold.
\end{theorem}

%Finally, we compute the components of the Ricci* tensor $\overline{\operatorname{Ric}}^{\ast}$ on $Z_-$  in order to prove the following estimates on the holomorphic scalar curvature $\overline{S}_J=\overline{S}-\overline{S}^{\ast}$, where $\overline{S}$ and $\overline{S}^{\ast}$ are the scalar curvatures of $\overline{\operatorname{Ric}}$ (the Ricci curvature of $g_t$) and $\overline{\operatorname{Ric}}^{\ast}$, respectively (see Section \ref{secQuadr} for the precise definitions)
%\begin{theorem}\label{TH:EstScalarz}
%	Let $(M,g)$ be an Einstein four-manifold with positive scalar curvature. Then,
%	on $(Z_-,g_t,J)$ the following inequality holds:
%	\[
%	-\dfrac{1}{2}|\nabla J|^2\leq\overline{S}-\overline{S}^{\ast}\leq|\nabla J|^2,
%	\]
%	Moreover, one of the equalities holds if and only if $M$ is self-dual.
%\end{theorem}
%Note that the Ricci* tensor measures, in some sense, 
%how much an almost Hermitian manifold is far from being K\"ahler.
\begin{rem}
It is a well-known fact (see Theorem \ref{muskar1}) that the twistor space $(Z_-, g_t, J)$ of an Einstein, self-dual manifold with positive scalar curvature $(M, g)$ is nearly-K\"ahler (indeed, K\"ahler); 
moreover, by Hitchin's classification result of K\"{a}hler twistor spaces
(see \cite{hit}) and 
of Einstein, self-dual manifolds with positive scalar curvature due to
Friedrich and Kurke (see \cite{frikur}), this is the case if and only if $(M, g)$  is isometric (up to quotients) to $\mathds{S}^4$ or $\mathds{CP}_2$.   
If $(M, g)$ is Einstein, but not necessarily self-dual, the properties of 
its  twistor space are not so very well understood: for some interesting 
results in this direction, see e.g. \cite{reznik} and \cite{finepanov}. 
Theorems \ref{TH:LocalQuadratic}, \ref{TH:LocalQuadratic2} show that the 
Atiyah-Hitchin-Singer almost Hermitian 
twistor space of an Einstein four-manifold bears a resemblance to a nearly 
K\"ahler manifold. Indeed, it is known that there exist relations between 
twistor and nearly K\"ahler geometries: for instance, in 1985, Belgun and 
Moroianu presented a homothetic classification of six-dimensional strict 
nearly K\"ahler manifolds whose canonical Hermitian connection has reduced 
holonomy by exploiting their twistorial structures (see \cite{belmoro}). Note 
that, in this work, we do not focus our attention on almost Hermitian 
manifolds (in particular, twistor space associated to a Riemannian manifold) 
satisfying these  ``weak''-nearly K\"ahler conditions. A natural question 
would be the following: is it possible to characterize the round metric on 
$\SS^4$ as the unique Einstein metric, by showing that the twistor space 
$(Z_-,g_t,J)$ of a four-sphere $\SS^4$ equipped with an Einstein metric 
$g_{Ein}$ cannot satisfy the conditions in Theorems \ref{TH:LocalQuadratic} 
and \ref{TH:LocalQuadratic2}, unless it is K\"ahler (or nearly K\"ahler)?

This will be the subject of future investigations, together with the analysis of higher order conditions on the almost complex structures and of curvature properties of $Z$.
\end{rem}

The paper is organized as follows: for the sake of completeness, and to fix the notation and conventions of the moving frame formalism, in Section \ref{secRiem} we recall some very well-known facts about the geometry of Riemannian $4$-manifolds. The short Section \ref{secTwist} is devoted to the formal definition of the twistor space $Z$ of a Riemannian manifold, with special attention to the case of dimension four. In Section \ref{secLin} we show that, given a Riemannian four-manifold $M$ and its twistor space $Z$, any linear condition on the covariant derivative of the almost complex structures on $Z$ implies that $M$ is self-dual; we also show how to recover some of the classical result (due to Atiyah, Hitchin and Singer and to  Mu\v{s}karov). In Section \ref{secQuadr} we focus on quadratic conditions, in order to study the case of Einstein four-manifolds whose metric is not necessarily self-dual. 
%We conclude the paper with four brief appendices, where we collect the components (in a local orthonormal coframe) of all the relevant geometric quantities involved in our analysis.

\

\section{Geometry of Riemannian four-manifolds}\label{secRiem}
In this section, for the sake of completeness,  we recall some useful and 
well known features of Riemannian four-manifolds (see e.g. \cite{besse}, 
\cite{friedr}, \cite{jenrig} 
\cite{salamon} and \cite{singthor}). Throughout the paper,
we adopt Einstein summation convention over repeated indices, unless it is 
specified otherwise.

\

\noindent{\bf The Hodge $\star$ operator in four dimensions.} Let $(M,g)$ a Riemannian oriented manifold of dimension $n$ and
let $\Lambda^k$ be the space (bundle) of the $k$-differential forms,
$1\leq k\leq n$. Given any local chart $(U,\varphi)$ that contains
$p\in M$, let $\{e_1,\ldots,e_n\}$ be a local, positively oriented,
orthonormal frame for $g$ on $U$ and let $\{\theta^1,\ldots,\theta^n\}$
be its dual orthonormal coframe, with $\theta^i\in\Lambda^1$, $\forall i=1,\ldots,n$.
Since $M$ is oriented, we can define a \emph{volume form} locally expressed by
\[
\omega=\theta^1\wedge\ldots\wedge\theta^n\in\Lambda^n.
\]
Now it is possible to define the \emph{Hodge $\star$ operator},
$\forall\, 1\leq k\leq n$, locally as
\begin{align} \label{hodge}\nonumber
\star:\Lambda^k&\longrightarrow\Lambda^{n-k}\\
\theta^{i_1}\wedge\ldots\wedge\theta^{i_k}&\longmapsto\star(\theta^{i_1}\wedge\ldots\wedge\theta^{i_k})=:\eta
\end{align}
where $\eta=\theta^{j_1}\wedge\ldots\wedge\theta^{j_{n-k}}\in\Lambda^{n-k}$ is the unique
$(n-k)$-form such that $(\theta^{i_1}\wedge\ldots\wedge\theta^{i_k})\wedge\eta=\omega$.
By construction, $\star$ satisfies the equation
\begin{equation} \label{hodinv}
\star^2=(-1)^{k(n-k)}I,
\end{equation}
where $I$ is the identity map from $\Lambda^k$ to itself.
Now, let $n=4$. We have that, if $k=2$, the $\star$ operator
is an involution: indeed, by definition and \eqref{hodinv},
\[
\star:\Lambda^2\longrightarrow\Lambda^2\quad \mbox{and} \quad \star^2=(-1)^{2\cdot2}I=I.
\]
If $\{\theta^1,\theta^2,\theta^3,\theta^4\}$ is
an orthonormal coframe for $M$ in a given chart, the set
$\{\theta^i\wedge\theta^j\}_{1\leq i<j\leq 4}$
is an orthonormal basis for $\Lambda^2$ with respect to the
inner product of differential forms induced on $\Lambda^2$
by the metric $g$.
Moreover, since $\star$ is an involution, its only two eigenvalues are $\pm1$
and it can be easily seen that
\begin{align}
\star(\theta^1\wedge\theta^2\pm\theta^3\wedge\theta^4)&=
\theta^3\wedge\theta^4\pm\theta^1\wedge\theta^2\notag ,\\
\star(\theta^1\wedge\theta^3\pm\theta^4\wedge\theta^2)&=
\theta^4\wedge\theta^2\pm\theta^1\wedge\theta^3, \label{eigens} \\
\star(\theta^1\wedge\theta^4\pm\theta^2\wedge\theta^3)&=
\theta^2\wedge\theta^3\pm\theta^1\wedge\theta^4\notag .
\end{align}
This means that
\begin{equation} \label{lamplusminus}
\Lambda_{\pm}:=span\big\{\theta^1\wedge\theta^2\pm\theta^3\wedge\theta^4,
\theta^1\wedge\theta^3\pm\theta^4\wedge\theta^2,
\theta^1\wedge\theta^4\pm\theta^2\wedge\theta^3 \big \}
\end{equation}
are the eigenspaces of $\star$ relative to the eigenvalue $\pm1$, respectively. Thus, by \eqref{lamplusminus}, we have that $\Lambda^2$ decomposes
in a direct sum of two three-dimensional subspaces (subbundles)
\begin{equation} \label{split}
\Lambda^2=\Lambda_+\oplus\Lambda_-.
\end{equation}
Note that, if $\{\theta^i\}_{i=1,\ldots,4}$ is a negatively oriented
orthonormal coframe, the signs $+$ and $-$ must be exchanged
in the right-hand side of \eqref{lamplusminus}. Moreover it is sufficient to define
\begin{align}
\alpha_{\pm}^1:=\frac{1}{\sqrt 2}(\theta^1\wedge\theta^2\pm\theta^3\wedge\theta^4) \notag \\
\alpha_{\pm}^2:=\frac{1}{\sqrt 2}(\theta^1\wedge\theta^3\pm\theta^4\wedge\theta^2) \label{ortheig}\\
\alpha_{\pm}^3:=\frac{1}{\sqrt 2}(\theta^1\wedge\theta^4\pm\theta^2\wedge\theta^3) \notag
\end{align}
to have orthonormal bases for $\Lambda_+$ and $\Lambda_-$, which are
called, respectively, the bundle of \emph{self-dual} and \emph{anti-self-dual}
$2$-forms of $M$. Clearly, any $2$-form $\eta$ can be written in a unique way as
\begin{equation} \label{sumforms}
\eta=\underbrace{\frac{1}{2}(\eta+\star\eta)}_{\in\Lambda_+}
+\underbrace{\frac{1}{2}(\eta-\star\eta)}_{\in\Lambda_-}=:\eta_+ +\eta_-,
\end{equation}
where $\eta_+$ is the \emph{self-dual} part of $\eta$ and
$\eta_-$ is the \emph{anti-self-dual} part.

There is an action of $SO(4)$ on $\Lambda^1$, defined as
\begin{align}
SO(4)\times\Lambda^1&\longrightarrow\Lambda^1 \label{action} \\
(a,\theta^i)&\longmapsto a(\theta^i):=(a^{-1})_j^i\theta^j, \notag
\end{align}
which induces an action of $SO(4)$ on $\Lambda^2$ given by
\begin{equation} \label{action2}
a(\theta^i\wedge\theta^j):=a(\theta^i)\wedge a(\theta^j)
\end{equation}
(see e.g. \cite{jenrig}). Moreover, it is known that $\mathfrak{so}(4)$ and $\Lambda^2$
are isomorphic \textit{via} the map
\begin{align}
f\colon\mathfrak{so}(4)&\longra\Lambda^2 \label{isomform} \\
X=(X_{ij})&\longmapsto\frac{1}{2}X_{ij}\theta^i\wedge\theta^j \notag
\end{align}
(here, $\mathfrak{so}(n)$ denotes the Lie algebra of $SO(n)$).
The isomorphism $f$ satisfies, for every $a\in SO(4)$, $X\in\mathfrak{so}(4)$,
\[
f(\operatorname{Ad}_a)(X)=a(f(X))
\]
where $\operatorname{Ad}$ means the adjoint representation of $SO(4)$, and,
since $\mathfrak{so}(4)\cong\mathfrak{so}(3)\oplus\mathfrak{so}(3)$,
$f$ induces isomorphisms $f_+,f_-$
\[
f_{\pm}:\mathfrak{so}(3)\longra\Lambda_{\pm}.
\]
The restriction of the adjoint action of $SO(4)$ to each copy
of $\mathfrak{so}(3)$ induces actions of $SO(3)$ on $\Lambda_+$
and $\Lambda_-$: namely, there exist smooth actions
\begin{align*}
SO(3)\times\Lambda_{\pm}&\longra\Lambda_{\pm}\\
(a,\eta_{\pm})&\longmapsto a(\eta_{\pm})
\end{align*}
such that, for every $a\in SO(3)$ and $Y\in\mathfrak{so}(3)$,
$a(f_{\pm}(Y))=f_{\pm}^{-1}(\operatorname{Ad}_a(Y))$.
Moreover, there exists a surjective Lie group homomorphism
\begin{equation} \label{mudef}
\mu:SO(4)\longra SO(3)\times SO(3)
\end{equation}
such that, for every $a\in SO(4)$, $\mu(a)=(a_+,a_-)$, where,
for every $\eta=\eta_++\eta_-\in\Lambda^2=\Lambda_+\oplus\Lambda_-$,
\[\
a(\eta)=a_+(\eta_+)+a_-(\eta_-).
\]

\

\noindent{\bf Decomposition of the Riemann curvature tensor.} Let $(M,g)$ be again a Riemannian, oriented, manifold of dimension $n$. We denote by $\mathrm{Riem}$ its Riemann curvature tensor and by $R_{ijkt}$ its
components with respect to an orthonormal coframe $\{\theta^1,\ldots,\theta^n\}$,
with $i,j,k,t=1,\ldots,n$. 
%which have the following symmetries:
%\begin{gather}
%	R_{ijkt}=-R_{jikt}=-R_{ijtk}, \notag \\
%	R_{ijkt}=R_{ktij} \label{riemsym}, \\
%	R_{ijkt}+R_{iktj}+R_{itjk}=0. \notag
%\end{gather}
We also define the \emph{curvature forms} $\Omega_j^i$ associated
to the orthonormal coframe $\{\theta^i\}$ as the 2-forms satisfying
the second Cartan structure equations
\begin{equation} \label{secondstruc}
d\theta_j^i=-\theta_k^i\wedge\theta_j^k+\Omega_j^i,
\end{equation}
where $\theta_j^i$ are the Levi-Civita connection 1-forms which
satisfy the first Cartan structure equations
\begin{equation} \label{firststruc}
d\theta^i=-\theta_j^i\wedge\theta^j
\end{equation}
(see e.g. \cite{cmBook}). Since, for every $i,j=1,\ldots,n$, $\theta_j^i+\theta_i^j=0$, we
have that $\Omega_j^i+\Omega_i^j=0$; thus, the matrix of the curvature
forms $\Omega$ takes values in $\mathfrak{so}(n)$. Moreover, the curvature
forms satisfy
\begin{equation} \label{omegariem}
	\Omega_j^i=\frac{1}{2}R_{ijkt}\theta^k\wedge\theta^t,
\end{equation}
where $R_{ijkt}$ are exactly the Riemann curvature tensor components with respect to $\{\theta^i\}$.

\noindent Let $e,\tilde{e}$ be orthonormal frames
such that there exists a smooth change $A:U\cap\widetilde{U}\longrightarrow O(m)$
for which $\tilde{e}=eA$ (i.e. $\tilde{e}_i=A_i^je_j$). Recall that,
if $\widetilde{\Omega}$ is the matrix of the curvature forms associated
to the frame $\tilde{e}$ (equivalently, to the coframe $\tilde{\theta}$
dual to $\tilde{e}$), then the following transformation law holds
(see \cite{alimasrig})
\begin{equation} \label{transfcurv}
	\widetilde{\Omega}=A^{-1}\Omega A.
\end{equation}

Let us define the \emph{Kulkarni-Nomizu product} $\KN$: if $\eta$
and $\kappa$ are two $(0,2)$-symmetric tensors, we have that
$\eta\KN\kappa$ is the $(0,4)$-tensor with components
\[
(\eta\KN\kappa)_{ijkt}:=\eta_{ik}\kappa_{jt}-\eta_{it}\kappa_{jk}+\eta_{jt}\kappa_{ik}-\eta_{jk}\kappa_{it}.
\]
It is well known that, $\forall n\geq 3$, the Riemann curvature tensor admits the decomposition
\begin{equation} \label{riemdec}
	\operatorname{Riem}=\weyl+\frac{1}{n-2}\operatorname{Ric}\KN g-\frac{S}{2(n-1)(n-2)}g\KN g,
\end{equation}
where $\operatorname{Ric}=R_{ij}\theta^i\otimes\theta^j$ is the Ricci curvature
tensor, $S$ is the scalar curvature and $\weyl$ is the \emph{Weyl tensor}. Equation \eqref{riemdec} can be written in local form as
\begin{equation} \label{riemdeccomp}
	{R_{ijkt}=W_{ijkt}+\frac{1}{n-2}(R_{ik}\delta_{jt}-R_{it}
		\delta_{jk}+R_{jt}\delta_{ik}-R_{jk}\delta_{it})-\frac{S}{(n-1)(n-2)}
		(\delta_{ik}\delta_{jt}-\delta_{it}\delta_{jk})}
\end{equation}
Now, let $n=4$. It is possible to rewrite the equations \eqref{riemdec}
and \eqref{riemdeccomp} thanks to \eqref{split}. We know that
$\{\theta^i\wedge\theta^j\}_{1\leq i<j\leq 4}$ is an orthonormal
basis for $\Lambda^2$ and that $\{\alpha_{\pm}^1,\alpha_{\pm}^2,\alpha_{\pm}^3\}$,
defined in \eqref{ortheig}, is an orthonormal basis for $\Lambda_{\pm}$,
respectively. The Riemann curvature tensor
corresponds to a symmetric operator, called the \emph{Riemann curvature operator},
defined as
\begin{equation} \label{curvop}
\mathcal{R}:\Lambda^2\longrightarrow\Lambda^2\qquad \mathcal{R}(\gamma)=\frac{1}{4}R_{ijkt}\gamma_{kt}\theta^i\wedge\theta^j=\frac{1}{2}\gamma_{kt}\Omega_t^k,
\end{equation}
where $\gamma_{kt}=\gamma(e_k,e_t)$ ($\{e_i\}$ is the orthonormal frame dual to $\{\theta^i\}$). Since \eqref{split} holds, every 2-form $\gamma$ can be written as
in \eqref{sumforms} and, since $\mathcal{R}(\gamma)\in\Lambda^2$,
it also can be expressed in a unique sum
\[
\mathcal{R}(\gamma)=\mathcal{R}(\gamma)_+ + \mathcal{R}(\gamma)_-.
\]
Evaluating the Riemann curvature operator on the bases
\eqref{ortheig} in order to find the self-dual and the anti-self-dual
parts of the images, we obtain that there exist three $3\times 3$ matrices,
$A=(A_{ij})$, $B=(B_{ij})$, $C=(C_{ij})$, $i,j=1,2,3$, such that, again with respect to the basis
$\{\alpha_{\pm}^1,\alpha_{\pm}^2,\alpha_{\pm}^3\}$ of $\Lambda_\pm$, the Riemann curvature
operator representative matrix takes the form
\begin{equation}\label{matrdeco}
\mathcal{R}=
\left(
\begin{array}{cc}
A & B^T \\
B & C
\end{array}
\right)
\end{equation}
where $A=A^T$, $C=C^T$ and $\operatorname{tr}A=\operatorname{tr}C=S/4$ (here,
$\operatorname{tr}$ denotes the matrix trace). More explicitly, if
$\gamma=c_j^+\alpha_+^j+c_k^-\alpha_-^k$,
\[
\mathcal{R}(\gamma)=(c_j^+A_{kj}+c_t^-B_{tk})\alpha_-^k+
(c_j^+B_{kj}+c_t^-C_{kt})\alpha_-^k.
\]
The explicit expressions for the entries of $A$, $B$, and $C$ in terms
of the Riemann tensor can be found in \cite{jenrig}.

Thus, we can think of $A$ (respectively, $C$) as a symmetric linear
map from $\Lambda_+$ (respectively, $\Lambda_-$) to itself, that is
$A\in \operatorname{End}(\Lambda_+)$, $C\in \operatorname{End}(\Lambda_-)$, and we can think of $B$
as a linear map from $\Lambda_+$ to $\Lambda_-$, i.e. $B\in \operatorname{Hom}(\Lambda_+,\Lambda_-)$.

%Since $\mathcal{R}(\theta^i\wedge\theta^j)=\Omega_j^i$, by the previous
%considerations we can write the matrix $\Omega$,
%thanks to the decomposition \eqref{matrdeco}, in the form
%\begin{equation} \label{tensorsum}
%	\Omega= A_{ij}\alpha_+^i\otimes\alpha_+^j+B_{ij}\alpha_-^i\otimes\alpha_+^j+
%	B_{ji}\alpha_+^i\otimes\alpha_-^j+C_{ij}\alpha_-^i\otimes\alpha_-^j,
%\end{equation}
%or, in matrix notation,
%\[
%\Omega=\alpha_+\otimes A\alpha_+ +\alpha_-\otimes B\alpha_+ +\alpha_+\otimes B^T\alpha_- +\alpha_-\otimes C\alpha_-.
%\]
It is also explicitly possible to write the transformation laws for $A$, $B$ and $C$.
Recall that, if $e,\tilde{e}$ are two orthonormal frames defined on $U$ and
$\widetilde{U}$ and $a:U\cap\widetilde{U}\longrightarrow SO(4)$
is a smooth change of frame, the equation \eqref{transfcurv} holds for $\Omega$.
Since, for every $a\in SO(4)$, $\mu(a)=(a_+,a_-)$ defines the restriction
of the action of $a$ on $\Lambda_+$ and $\Lambda_-$,
we obtain the following transformation laws
\begin{equation} \label{abctrans}
	\widetilde{A}=a_+^{-1} Aa_+, \qquad \widetilde{B}=a_-^{-1} Ba_+, \qquad \widetilde{C}=a_-^{-1} Ca_-.
\end{equation}
Since it will simplify all our next computations, we introduce the (purely local) quantities

\begin{equation} \label{qdef}
  \qd{ab} := R_{12ab}+R_{34ab}; \qquad \qt{ab} := R_{13ab}+R_{42ab}; \qquad \qq{ab} := R_{14ab}+R_{23ab}.
\end{equation}
Note that, by the first Bianchi identity, we deduce
\begin{equation}
  \qt{12}+\qt{34} = \qd{13}+\qd{42}; \quad \qq{12}+\qq{34}= \qd{14}+\qd{23}; \quad \qt{14}+\qt{23} = \qq{13}+\qq{42}.
\end{equation}
Rewriting the components listed in \cite{jenrig}, 
we have the following expressions for $A$ and $B$:
\noindent
\begin{multicols}{2}
\noindent
\begin{align*}
  A_{11} &=
  %\frac{1}{2}\pa{R_{1212}+2R_{1234}+R_{3434}} =
  \frac{1}{2}\pa{\qd{12}+\qd{34}}; \\ 
  A_{22} &=
  %\frac{1}{2}\pa{R_{1313}+2R_{1342}+R_{4242}} = 
  \frac{1}{2}\pa{\qt{13}+\qt{42}}; \\ 
  A_{33} &= 
  %\frac{1}{2}\pa{R_{1414}+2R_{1423}+R_{2323}} = 
  \frac{1}{2}\pa{\qq{14}+\qq{23}};
\end{align*}
\begin{align*}
  A_{12} &=A_{21}= 
  %\frac{1}{2}\pa{R_{1213}+R_{1334}+R_{1242}+R_{3442}} = 
  \frac{1}{2}\pa{\qt{12}+\qt{34}} = \frac{1}{2}\pa{\qd{13}+\qd{42}} \\ 
  A_{13} &=A_{31} = 
  %\frac{1}{2}\pa{R_{1214}+R_{1434}+R_{1223}+R_{2334}} = 
  \frac{1}{2}\pa{\qq{12}+\qq{34}} = \frac{1}{2}\pa{\qd{14}+\qd{23}}; \\ A_{23} &=A_{32} = 
  %\frac{1}{2}\pa{R_{1314}+R_{1442}+R_{1323}+R_{2342}} = 
  \frac{1}{2}\pa{\qt{14}+\qt{23}} = \frac{1}{2}\pa{\qq{13}+\qq{42}};
\end{align*}
\\
\end{multicols}
\noindent
\begin{multicols}{2}
\noindent
\begin{align*}
  B_{11} &= \frac{1}{2}\pa{\qd{12}-\qd{34}}; \\ B_{22} &= \frac{1}{2}\pa{\qt{13}-\qt{42}}; \\ B_{33} &= \frac{1}{2}\pa{\qq{14}-\qq{23}}; \\B_{12} &= \frac{1}{2}\pa{\qt{12}-\qt{34}}; \\ B_{21} &= \frac{1}{2}\pa{\qd{13}-\qd{42}}; 
\end{align*}
\begin{align*}  
  B_{13} &= \frac{1}{2}\pa{\qq{12}-\qq{34}}; \\ B_{31} &= \frac{1}{2}\pa{\qd{14}-\qd{23}}; \\ B_{23} &= \frac{1}{2}\pa{\qq{13}-\qq{42}}; \\B_{32} &= \frac{1}{2}\pa{\qt{14}-\qt{23}}.
\end{align*}
\end{multicols}
As far as the Weyl tensor is concerned, by \eqref{riemdeccomp} and \eqref{matrdeco}, we have
\begin{flalign*}
W_{1212}&=\frac{1}{2}\left(A_{11}-\frac{S}{12}\right)+
\frac{1}{2}\left(C_{11}-\frac{S}{12}\right)\\
W_{1213}&=\frac{1}{2}A_{12}+\frac{1}{2}C_{12}, & W_{1214}=\frac{1}{2}A_{13}+\frac{1}{2}C_{13},&\\
W_{1223}&=\frac{1}{2}A_{13}-\frac{1}{2}C_{13},
& W_{1242}=\frac{1}{2}A_{12}-\frac{1}{2}C_{12},\\
W_{1234}&=\frac{1}{2}\left(A_{11}-\frac{S}{12}\right)-
\frac{1}{2}\left(C_{11}-\frac{S}{12}\right),
& W_{1313}=\frac{1}{2}\left(A_{22}-\frac{S}{12}\right)+
\frac{1}{2}\left(C_{22}-\frac{S}{12}\right),
\end{flalign*}
\begin{flalign*}
W_{1314}&=\frac{1}{2}A_{23}+\frac{1}{2}C_{23},
& W_{1323}=\frac{1}{2}A_{23}-\frac{1}{2}C_{23},\\
W_{1342}&=\frac{1}{2}\left(A_{22}-\frac{S}{12}\right)-
\frac{1}{2}\left(C_{22}-\frac{S}{12}\right),
& W_{1334}=\frac{1}{2}A_{12}-\frac{1}{2}C_{12},\\
W_{1414}&=\frac{1}{2}\left(A_{33}-\frac{S}{12}\right)+
\frac{1}{2}\left(C_{33}-\frac{S}{12}\right),
& W_{1423}=\frac{1}{2}\left(A_{33}-\frac{S}{12}\right)-
\frac{1}{2}\left(C_{33}-\frac{S}{12}\right),
\end{flalign*}
\begin{flalign*}
W_{1442}&=\frac{1}{2}A_{23}-\frac{1}{2}C_{23},
& W_{1434}=\frac{1}{2}A_{13}-\frac{1}{2}C_{13},\\
W_{2323}&=\frac{1}{2}\left(A_{33}-\frac{S}{12}\right)+
\frac{1}{2}\left(C_{33}-\frac{S}{12}\right),
& W_{2342}=\frac{1}{2}A_{23}+\frac{1}{2}C_{23},\\
W_{2334}&=\frac{1}{2}A_{13}+\frac{1}{2}C_{13},
& W_{4242}=\frac{1}{2}\left(A_{22}-\frac{S}{12}\right)+
\frac{1}{2}\left(C_{22}-\frac{S}{12}\right),\\
W_{3442}&=\frac{1}{2}A_{12}+\frac{1}{2}C_{12},
& W_{3434}=\frac{1}{2}\left(A_{11}-\frac{S}{12}\right)+
\frac{1}{2}\left(C_{11}-\frac{S}{12}\right).
\end{flalign*}
It is apparent that all the components can be written as a sum of two addends
\[
W_{ijkt}=W_{ijkt}^+ + W_{ijkt}^-
\]
This means that the Weyl tensor $W$ splits into a sum of two $(0,4)$-tensors
\[
W=W^+ + W^-
\]
called, respectively, the \emph{self-dual} and the \emph{anti-self-dual}
components of $W$. A four-dimensional Riemannian manifold is called \emph{self-dual}
(respectively \emph{anti-self-dual}) if $W^-=0$ (resp., $W^+=0$). By a direct check
of the entries of $A$, $B$ and $C$ and the coefficients of $W^+$ and $W^-$,
we easily obtain the following
\begin{theorem} \label{cond4mani}
	Let $(M,g)$ be a Riemannian manifold of dimension $4$. Then,
	\begin{itemize}
		\item $M$ is self-dual if and only if $C-\frac{S}{12}I_3=0$ for
		every orthonormal positively oriented coframe (respectively, if and only if $A-\frac{S}{12}I_3=0$
		for every orthonormal negatively oriented coframe);
		\item $M$ is anti-self-dual if and only if $A-\frac{S}{12}I_3=0$ for every
		orthonormal positively oriented coframe (respectively, if and only if $C-\frac{S}{12}I_3=0$
		for every orthonormal negatively oriented coframe);
		\item $M$ is Einstein if and only if $B=0$ for every orthonormal positively or negatively oriented coframe.
	\end{itemize}
\end{theorem}

\

\section{The twistor space of a four-manifold}\label{secTwist}

Let $(M,g_M)$ be a connected Riemannian manifold of dimension $2m$. We define
its \emph{twistor space} $Z$ associated to $M$
as the set of the pairs $(p,J_p)$, where $p\in M$ and
$J_p$ is a $g$-orthogonal complex structure on $T_pM$. It is not hard
to show that the set of all $g$-orthogonal complex structures is
diffeomorphic to $O(2m)/U(m)$, where
\[
U(m):=\{A\in O(2m):A^TJ_m=J_mA\}
\]
and $J_m$ is a matrix in $O(2m)\cap\mathfrak{so}(2m)$ with entries
$(J_m)_{kl}=\delta_k^{l+1}-\delta_l^{k+1}$ (see, for instance, 
\cite{debnan}); therefore, it can be shown
that, if we denote as $O(M)$ as the orthonormal frame bundle of $M$,
$Z$ is the associated bundle
\[
Z=O(M)\times_{O(2m)}(O(2m)/U(m))=O(M)/U(m).
\]
This means that
there exists a surjective smooth map $\sigma:O(M)\longra Z$ such that
$\sigma$ defines a principal bundle $(O(M),Z,U(m))$ with structure group
$U(m)$. Moreover, the map
\begin{align*}
\pi_Z:Z&\longra M\\
(p,J_p)&\longmapsto p
\end{align*}
determines a fiber bundle $(Z,M,O(2m)/U(m),O(2m))$ with structure group
$O(2m)$ and standard fiber $O(2m)/U(m)$ (see \cite{kobnom1}).

Now, observe that, if $M$ is oriented, the orthonormal frame bundle is
not connected: indeed, in this case $O(M)$ has two connected components
\[
O(M)=O(M)_+\sqcup O(M)_-,
\]
where $O(M)_+$ (resp., $O(M)_-$) is the bundle of positively (resp.,
negatively) oriented frames on $M$.
As a consequence, $Z$ itself has two connected components 
\[
Z_{\pm}=\bigslant{O(M)_{\pm}}{U(m)}=\bigslant{SO(M)}{U(m)},
\]
where $SO(M)$ denotes the orthonormal
oriented frame bundle over $M$; moreover, one can obviously define
the bundle projections 
\[
\pi_{Z_{\pm}}:Z_{\pm}\longra M \quad \mbox{ and } \quad
\sigma_{\pm}:O(M)_{\pm}\longra Z_{\pm}.
\]
In accordance to much of the literature (see, for instance, 
\cite{salamon}), by convention
we choose $Z_-$ to be the twistor space
associated to a Riemannian manifold $(M,g)$; therefore, from now on,
$Z=Z_-$, $\pi_Z=\pi_{Z_-}$ and $\sigma=\sigma_-$. 

It is known that, in general, there exists a one-parameter family
of Riemannian metrics $g_t$ on $Z$, with $t>0$, constructed as 
pullbacks \emph{via} $\sigma$ 
of the unique (up to multiplication by a positive constant) 
$O(2m)$-invariant
Riemannian metric on $O(M)$ which is also horizontal with respect
to $\sigma$
(see \cite{debnan}, \cite{frikur}
and \cite{jenrig}): as a consequence, the map $\sigma$ becomes a Riemannian
submersion and the fibers are totally geodesic submanifolds of 
$O(M)$. 

Let $m=2$; from now on, we adopt the index
conventions $1\leq a,b,c,\ldots\leq 4$ and $1\leq p,q,\ldots,\leq 6$.
Given a local orthonormal coframe $\{\omega^a\}_{a=1,\ldots,4}$ on
an open set $U\subset M$, with Levi-Civita connection forms $\{\omega_b^a\}$,
we define
\begin{equation*}
	\omega^5:=\frac{1}{2}\pa{\omega^1_3+\omega^4_2}, \quad
	\omega^6:=\frac{1}{2}\pa{\omega^1_4+\omega^2_3};
\end{equation*}
a local orthonormal coframe on $(Z,g_t)$ is obtained by considering
the pullbacks of $\omega^1,\ldots,\omega^6$ \emph{via} a smooth local section
$u:U\longra O(M)$ of the principal $U(m)$-bundle defined by $\sigma$, where
$U$ is a suitable open subset of $Z$.
This means that
\begin{equation} \label{metrictwist}
	g_t=\sum_{p=1}^6(\theta^p)^2,
\end{equation}
where
\begin{equation} \label{cofrtwist}
	\theta^a:=u^{\ast}(\omega^a), \quad \theta^5:=2tu^{\ast}(\omega^5), \quad
	\theta^6:=2tu^{\ast}(\omega^6);
\end{equation}
for the sake of simplicity, we write $\omega^a$ for $u^{\ast}(\omega^a)$ and
similarly for $2t\omega^5$ and $2t\omega^6$. By \eqref{metrictwist} and
\eqref{cofrtwist}, we can write
\begin{equation} \label{metrictwist2}
g_t = g_M + 4t^2 \sq{\pa{\omega^5}^2+\pa{\omega^6}^2}=g_M + 
(\theta^5)^2+(\theta^6)^2
\end{equation}
(again the pullback notation is omitted). In order to compute the Levi-Civita connection forms $\theta_q^p$
and the curvature forms $\Theta_q^p$ for
the orthonormal coframe defined in \eqref{cofrtwist}, recall the
structure equations \eqref{firststruc} and \eqref{secondstruc}.
By direct computation, we obtain (see \cite{jenrig})
\begin{align} \label{connformtwist}
	\theta^a_b &= \omega^a_b + \frac{1}{2}t\pa{\qt{ba}\theta^5+\qq{ba}\theta^6}, \\ \theta^5_b &=\frac{1}{2}t\qt{ba}\theta^a, \quad \theta^6_b= \frac{1}{2}t\qq{ba}\theta^a; \\ \theta^5_6 &= \omega^1_2+\omega^3_4.
\end{align}
Now, let us denote as $\overline{\operatorname{Riem}}$,
$\overline{\ricc}$ and $\overline{S}$ the Riemann curvature tensor,
the Ricci tensor and the scalar curvature of $(Z,g_t)$, respectively;
by \eqref{secondstruc} and \eqref{omegariem}, it is easy to obtain the
coefficients $\overline{R}_{pqrs}$, $\overline{R}_{pq}$ and the
scalar curvature $\overline{S}$ in terms of $\qd{ab}$, $\qt{ab}$
and $\qq{ab}$ (see \cite{cdmtwist}).
%Moreover, denoting as $\overline{\operatorname{W}}$
%the Weyl tensor of $(Z,g_t)$, by \eqref{riemdeccomp} we obtain the components $\overline{W}_{pqrs}$ of $\overline{\operatorname{W}}$
%(see \eqref{riemtwist}, \eqref{ricctwist} and \eqref{scaltwist}).

It is possible to introduce two almost complex structures on $(Z,g_t)$.
Indeed, let $\{\theta^p\}$ be the orthonormal coframe defined in
\eqref{cofrtwist} and $\{e_p\}$ be its dual frame; then,
we define
\begin{equation} \label{twistalmcomp}
	J_{\pm}=\theta^1\otimes e_2 - \theta^2\otimes e_1 +
	\theta^3\otimes e_4 - \theta^4\otimes e_3 \pm\theta^5\otimes e_6
	\mp\theta^6\otimes e_5.
\end{equation}
It is clear that $J_+$ and $J_-$, which were introduced by
Atiyah, Hitchin and Singer (\cite{athisin}) and by Eells and Salamon
(\cite{elsal}), respectively,
are $g_t$-orthogonal almost
complex structures on $(Z,g_t)$, that is, $(Z,g_t,J_{\pm})$ is
an almost Hermitian manifold. We can write the local expression of the almost
complex structures with respect to \eqref{cofrtwist} and the
dual orthonormal frame as
\begin{align*}
	J_{\pm}&=(J_{\pm})_p^q\theta^p\otimes e_q,
\end{align*}
where
\[
(J_+)_1^2=(J_+)_3^4=(J_+)_5^6=(J_-)_1^2=(J_-)_3^4=-(J_-)_5^6=1,
\]
and all other components are zero.
%It is easy to see that the K\"{a}hler forms of $(Z,g_t,J)$ and
%$(Z,g_t,\JJ)$ are, respectively,
%\begin{align}
%	\omega_{\pm}=\theta^1\wedge\theta^2+\theta^3\wedge\theta^4\pm
%	\theta^5\wedge\theta^6;
%\end{align}
%%(see the appendix \ref{appb} for the expressions of
%%$d\omega_+$, $d\omega_-$, $\delta\omega_+$
%%and $\delta\omega_-$, that is equations \eqref{kahtwistplus}, \eqref{kahtwistmin} and \eqref{codifftwist}).
%therefore, one can compute the explicit expression for 
%the exterior derivatives $d\omega_{\pm}$ and the codifferentials
%$\delta\omega_{\pm}$.

We can compute the components 
of $\nabla J_+$ and $\nabla J_-$ with respect to the coframe defined in 
\eqref{cofrtwist}, recalling that
\begin{equation} \label{nablajcomp}
	\nabla J_{\pm}=(J_{\pm})_{q,t}^p\theta^t\otimes\theta^q\otimes e_p,
\end{equation}
where 
\[
(J_{\pm})_{q,t}^p\theta^t=d(J_{\pm})_q^p-
(J_{\pm})_s^p\theta_q^s+(J_{\pm})_q^s\theta_s^p;
\]
the components derived from \eqref{nablajcomp}, which depend on the
curvature of $(M,g)$, 
will be central in the proof of 
Theorems \ref{linselfdual} and \ref{linselfdual2} (for the explicit
expression of the components in terms of \eqref{qdef}, see \cite{cdmtwist}). 

The almost complex structures $J_+$ and $J_-$ induce the corresponding
Nijenhuis tensors $N_{J_+}$ and $N_{J_-}$, locally defined by
\begin{equation} \label{nijen}
	N_{J_{\pm}} = (N_{\pm})^p_{tq}\theta^t\otimes\theta^q\otimes e_p, \quad (N_{\pm})^p_{tq}=-(N_{\pm})^p_{qt},
\end{equation}
where
\begin{equation}\label{nijenComp}
	(N_{\pm})_{pq}^r=(J_{\pm})_p^s(J_{\pm})_{s,q}^r-(J_{\pm})_q^s(J_{\pm})_{s,p}^r+(J_{\pm})_q^s(J_{\pm})_{p,s}^r-(J_{\pm})_p^s(J_{\pm})_{q,s}^r;
\end{equation}
we recall that, by the well-known Newlander-Nirenberg Theorem 
(\cite{newlnir}), the vanishing of the Nijenhuis tensor is equivalent to
the integrability of the almost complex structure, which, in this case,
is induced by a holomorphic atlas of charts. 

As said in the introduction, 
the problem of the integrability of $J_+$ and $J_-$ was solved by Atiyah, 
Hitchin and Singer (\cite{athisin}) and by Eells and Salamon (\cite{elsal}),
respectively: indeed, $J_+$ is integrable if and only if $(M,g)$ is a self-dual
manifold, while $J_-$ is never integrable. However, in the same paper 
\cite{elsal} the authors proved a correspondence between minimal 
$2$-dimensional submanifolds of $(M,g)$ and $(J_-)$-holomorphic curves in
$Z$. 
%, and their covariant
%derivatives $\nabla N_J$ and $\nabla N_{\JJ}$

For the explicit expressions of \eqref{nablajcomp} and \eqref{nijen} in terms of $\qd{ab}$, $\qt{ab}$ and 
$\qq{ab}$, we refer the reader to \cite{cdmtwist}.

\

\section{Linear conditions on $J_+$ and $J_-$: proof of Theorem \ref{TH:GeneralLinear}}\label{secLin}

In this section we show that, given a Riemannian
four-manifold $M$ and its twistor space $Z$, any linear condition
on the covariant derivative of the almost complex structures $J_+$ and
$J_-$ on $Z$ implies that $M$ is self-dual. From now on, we
will constantly make use of the components of $\nabla J_{\pm}$
and $N_{J_{\pm}}$ listed in \cite{cdmtwist}.

Let us start with the following proposition, which should be compared with Theorem \ref{cond4mani}:
\begin{proposition} \label{suffcond}
	Let $(M,g)$ be an oriented Riemannian four-manifold. Let
	\[
	\mathcal{R}=
	\left(
	\begin{array}{cc}
	A & B^T\\
	B & C
	\end{array}
	\right)
	\]
	be the decomposition of the Riemann curvature operator on $M$,
	with $A=(A_{ij})_{i,j=1,2,3}$, $B=(B_{ij})_{i,j=1,2,3}$
	and $C=(C_{ij})_{i,j=1,2,3}$.
	Then,
	\begin{enumerate}
		\item $M$ is self-dual if and only if,
		for every negatively oriented local orthonormal coframe, $A_{ij}=0$ for some
		$i\neq j$ or $A_{kk}=A_{ll}$ for some $k,l$;
		\item $M$ is Einstein if and only if, for
		every negatively oriented local orthonormal coframe, $B_{ij}=0$ for some
		$i,j$.
	\end{enumerate}
\end{proposition}
\begin{proof}
Recall the transformation laws for $A$, $B$ and $C$ defined in
\eqref{abctrans} and the surjective Lie group homomorphism \eqref{mudef}.

$(1)$ If $M$ is self-dual, then $A$ is a scalar matrix with
	$A_{ij}=\frac{S}{12}\delta_{ij}$. Conversely, let us prove the claim
	for $A_{12}=0$ and $A_{11}=A_{22}$, since the other cases can be shown in
	an analogous way. If, for every negatively oriented orthonormal coframe,
	$A_{12}=0$, then the matrix $A$ has the form
	\[
	A=
	\pmat{
		A_{11} & 0 & A_{13}\\
		0 & A_{22} & A_{23}\\
		A_{13} & A_{23} & A_{33}
	}
	\]
	on $O(M)_-$. Equivalently, for every smooth change of frames
	$a:U\longra SO(4)$, the transformed matrix $\tilde{A}$ is such that
	$\tilde{A}_{12}=0$. Thus, let us choose $a\in SO(4)$ such that
	$\mu(a)=(a_+,a_-)$, where
	\[
	a_+=
	\pmat{
		0 & 0 & -1\\
		0 & 1 & 0\\
		1 & 0 & 0
	}.
	\]
	We have that
	\[
	\tilde{A}=a_+^{-1}Aa_+=
	\pmat{
		A_{33} & A_{23} & -A_{13}\\
		A_{23} & A_{22} & 0\\
		-A_{13} & 0 & A_{11}
	};
	\]
	thus, $\tilde{A}_{12}=A_{23}=0$ on $O(M)_-$, that is, $A$ has the form
	\[
	A=
	\pmat{
		A_{11} & 0 & A_{13}\\
		0 & A_{22} & 0\\
		A_{13} & 0 & A_{33}
	}.
	\]
	Repeating the argument on $A$ with
	\[
	a_+=
	\pmat{
		1 & 0 & 0\\
		0 & 0 & -1\\
		0 & 1 & 0
	},
	\]
	we obtain that $A_{13}=0$; hence, $A$ is a diagonal matrix. Choosing
	other suitable changes of frames, it is not hard to show that
	$A_{11}=A_{22}=A_{33}$, i.e. $A$ is a scalar matrix; by Theorem \eqref{cond4mani},
	$M$ is self-dual.
	
	Similar computations show that, if $A_{11}=A_{22}$ on $O(M)_-$, then
	$A$ is a scalar matrix, i.e. $M$ is self-dual.

\smallskip

$(2)$ If $M$ is Einstein, then $B=0$ by Theorem \eqref{cond4mani}.
	Conversely, suppose, for instance,
	that $B_{11}=0$ on $O(M)_-$ (the other cases can be proved analogously).
	Again, this means that, for every change of frames $a$, the transformed
	matrix $\tilde{B}$ is such that $\tilde{B}_{11}=0$. Let us choose
	$a\in SO(4)$ such that $\mu(a)=(a_+,a_-)$, where
	\[
	a_-=I_3, \qquad
	a_+=
	\pmat{
		0 & -1 & 0\\
		1 & 0 & 0\\
		0 & 0 & 1
	}.
	\]
	Hence, we have that
	\[
	\tilde{B}=a_-^{-1}Ba_+=
	\pmat{
		B_{12} & 0 & B_{13}\\
		B_{22} & -B_{21} & B_{23}\\
		B_{32} & -B_{31} & B_{33}
	};
	\]
	this implies that $\tilde{B}_{11}=B_{12}=0$. By the same argument,
	if we choose
	\[
	a_+=
	\pmat{
		0 & 0 & -1\\
		0 & 1 & 0\\
		1 & 0 & 0
	},
	\]
	we conclude that $B_{13}=0$. Now, repeating the argument with
	suitable choices of $a_+$ and $a_-$, we obtain
	\[
	B_{21}=B_{22}=B_{23}=B_{31}=B_{32}=B_{33}=0,
	\] \\
	that is, $B=0$. Therefore, $M$ is Einstein by Theorem \eqref{cond4mani}.
\end{proof}
The first main result of this section is the following (see Theorem \ref{TH:GeneralLinear} in the introduction)
\begin{theorem} \label{linselfdual}
Let $(M,g)$ be a Riemannian four-manifold and let $(Z,g_t,J_+)$ be its twistor space.
If, for every $X,Y\in\mathfrak{X}(Z)$,
\begin{align}  \label{lincond}
&a_1(\nabla_X J_+)Y+a_2(\nabla_Y J_+)X+a_3(\nabla_{J_+X} J_+)Y+a_4(\nabla_{J_+Y} J_+)X+
a_5(\nabla_{J_+X} J_+)J_+Y+\\
&+a_6(\nabla_{J_+Y} J_+)J_+X+a_7(\nabla_{X} J_+)J_+Y+a_8(\nabla_{Y} J_+)J_+X=0 \notag
\end{align}
for some $a_i\in\erre$, $i=1,\ldots,8$, such that $a_j\neq 0$ for some $j$, then,
$M$ is self-dual.
\end{theorem}
\begin{proof}
For the sake of simplicity, we will denote $(J_+)_{q,t}^p=J_{q,t}^p$.
First, we rewrite the equality in \eqref{lincond} with respect to a local orthonormal
frame $\{e_t\}_{t=1,\ldots,6}$ by putting $X=e_p, Y=e_q$. This implies that
\begin{align} \label{lincondort}
a_1J_{q,p}^t+a_2J_{p,q}^t+a_3J_p^rJ_{q,r}^t+a_4J_q^rJ_{p,r}^t+a_5J_p^rJ_q^sJ_{s,r}^t+
a_6J_q^rJ_p^sJ_{s,r}^t+a_7J_q^sJ_{s,p}^t+a_8J_p^sJ_{s,q}^t=0
\end{align}
for every $p,q,t=1,\ldots,6$. We now proceed by steps.

$(1)$ We start by considering \eqref{lincondort} with $p=5,q=2,t=1$, i.e.
	\[
	(a_8-a_4)\qt{12}-(a_2+a_6)\qq{12}=0.
	\]
	Putting $p=5,q=4,t=3$, we easily obtain
	\[
	(a_8-a_4)\qt{34}-(a_2+a_6)\qq{34}=0;
	\]
	Summing these two equalities, we can write
	\[
	(a_8-a_4)A_{12}-(a_2+a_6)A_{13}=0.
	\]
	Repeating the argument with $p=6,q=2,t=1$ and $p=6,q=4,t=3$ we have that
	\[
	(a_2+a_6)A_{12}+(a_8-a_4)A_{13}=0.
	\]
	Thus, we deduce the following system of equations:
	\[
	\begin{cases}
	(a_8-a_4)A_{12}-(a_2+a_6)A_{13}=0,\\
	(a_2+a_6)A_{12}+(a_8-a_4)A_{13}=0.
	\end{cases}
	\]
	If at least one of the coefficients $(a_8-a_4)$ and $(a_2+a_6)$ is
	different from $0$, we must have that $A_{12}=A_{13}=0$ on $O(M)_-$, since
	\eqref{lincond} is a global condition. By Proposition \ref{suffcond},
	$M$ is self-dual.
	
	Note that, if we exchange the values of $p$ and $q$ in all the previous calculations,
	the following system holds:
	\[
	\begin{cases}
	(a_7-a_3)A_{12}-(a_1+a_5)A_{13}=0,\\
	(a_1+a_5)A_{12}+(a_7-a_3)A_{13}=0.
	\end{cases}
	\]
	As before, if at least one of the coefficients $(a_1+a_5)$ and $(a_7-a_3)$ is different
	from zero, then $M$ is self-dual.

\smallskip

$(2)$ Now, we have to study the case
	\[
	a_1=-a_5,\quad a_2=-a_6,\quad a_3=a_7,\quad a_4=a_8,
	\]
	that is
	\[
	a_1(J_{q,p}^t-J_p^rJ_q^sJ_{s,r}^t)+a_2(J_{p,q}^t-J_q^rJ_p^sJ_{s,r}^t)
	+a_3(J_p^rJ_{q,r}^t+J_q^sJ_{s,p}^t)+a_4(J_q^rJ_{p,r}^t+J_p^sJ_{s,q}^t)=0.
	\]
	By choosing $p=5,q=3,t=1$ and $p=6,q=3,t=1$, we obtain
	\[
	\begin{cases}
	(a_1+a_2)(A_{22}-A_{33})+2(a_3+a_4)A_{23}=0,\\
	(a_3+a_4)(A_{22}-A_{33})-2(a_1+a_2)A_{23}=0.
	\end{cases}
	\]
	Again, if not all the coefficients vanish, the system holds if and only
	if $A_{22}=A_{33}$ and $A_{23}=0$. By Proposition \ref{suffcond}, $M$ is self-dual.
	\item Finally, we have to show the claim when
	\[
	a_1=-a_2,\quad a_3=-a_4.
	\]
	Choosing $p=3,q=1,t=5$ and $p=4,q=1,t=5$, we get the system
	\[
	\begin{cases}
	a_1(A_{22}-A_{33})+2a_3A_{23}=0,\\
	-a_3(A_{22}-A_{33})+2a_1A_{23}=0.
	\end{cases}
	\]
	Since, by hypothesis, at least one of the coefficients does not vanish,
	we conclude that $A_{22}=A_{33}$ and $A_{23}=0$, i.e. $M$ is self-dual.
\end{proof}
The previous theorem allows us to easily prove a well-known result, due
to Atiyah, Hitchin and Singer (\cite{athisin}):
\begin{theorem} \label{atiyahhitchsing}
Let $(M,g)$ be a Riemannian four-manifold and let $(Z,g_t,J_+)$ be its twistor space. Then,
the almost complex structure $J_+$ is integrable if and only if $M$ is self-dual.
\end{theorem}
\begin{proof}
Recall that an almost complex structure is integrable if and only if the associated
Nijenhuis tensor identically vanishes. Thus, by direct inspection of the components,
if $M$ is self-dual, the Nijenhuis tensor $N_{J_+}$ is identically null. Conversely,
note that the condition of integrability for $J_+$ corresponds to the equation \eqref{lincond}
with
\[
a_1=a_2=a_5=a_6=0, \qquad a_4=a_8=-a_3=-a_7=1.
\]
Thus, if $N_{J_+}$ vanishes identically, then $M$ is self-dual.
\end{proof}
\begin{rem}
	\begin{enumerate}
	\item We point out that, as mentioned by Apostolov, Grantcharov and 
	Ivanov (\cite{apoherm}), for every point $p\in M$ 
	such that $\weyl^-\neq 0$
	at $p$ with non-degenerate spectrum, 
	there exist exactly two almost complex structures
	$J_1$ and $J_2$, determined up to sign, such that $N_{J_+}$
	vanishes at $(p,J_1)$, $(p,J_2)\in\pi_Z^{-1}(p)$ (see also
	\cite{apogaud}, \cite{salfour} and \cite{salorth}): in other words,
	by a direct computation of the components of $N_{J_+}$, 
	for every such $p\in M$ there exist two negatively
	oriented orthonormal frames
	$e_1$ and $e_2$ such that, with respect to those, 
	$A_{22}=A_{33}$ and $A_{23}=0$.
	\item A generalization of Theorem \ref{atiyahhitchsing} involving
	the divergences of the Nijenhuis tensors was
	obtained in \cite{cdmtwist} (Theorem 5.5).
	\end{enumerate}
\end{rem}
In the same spirit, one can provide an alternative proof for the
characterization results due to Mu\v{s}karov (\cite{mus}), which generalize 
a Theorem due to Friedrich and Kurke (\cite{frikur}): for instance, we
can show the following
\begin{theorem} \label{muskar1}	
	$(Z,g_t,J_+)$ is a $q^2$-\emph{K\"{a}hler manifold}, i.e., for
	every $X\in\mathfrak{X}(Z)$,
	\begin{equation} \label{qqkahl}
		(\nabla_XJ_+)X+(\nabla_{J_+X}J_+)J_+X=0,
	\end{equation}
	if and only if
	$M$ is Einstein, self-dual, with positive scalar curvature
	equal to $12/t^2$. In particular, this holds if and only if 
	$(Z,g_t,J_+)$ is a K\"{a}hler manifold. 
\end{theorem}
\begin{rem}
	Note that \eqref{qqkahl} is satisfied on any nearly K\"{a}hler manifold
	and on any almost K\"{a}hler manifold.
\end{rem}
\begin{proof}
	One direction is trivial: indeed, if $(M,g)$ is Einstein, self-dual
	with positive scalar curvature, then $(Z,g_t,J_+)$ is a K\"{a}hler
	manifold (\cite{frikur}), which implies that it is also 
	$q^2$-K\"{a}hler. 
	
	Conversely, let us suppose that $(Z,g_t,J_+)$ is 
	$q^2$-K\"ahler.
	Note that we can rewrite \eqref{qqkahl} as
	\[
	(\nabla_X J_+)Y+(\nabla_Y J_+)X+(\nabla_{J_+X} J_+)J_+Y+(\nabla_{J_+Y} J_+)J_+X=0
	\]
	for every $X,Y\in\mathfrak{X}(Z)$, which is \eqref{lincond} with
	\[
	a_3=a_4=a_7=a_8=0, \qquad a_1=a_2=a_5=a_6.
	\]
	This immediately implies that $M$ is self-dual. With some
	straightforward calculation using \eqref{qqkahl}, we obtain the system
	\[
	\begin{cases}
	B_{22}+B_{33}=\dfrac{S}{6}-\dfrac{2}{t^2},\\
	B_{33}-B_{22}=-\dfrac{S}{6}+\dfrac{2}{t^2}.
	\end{cases}
	\]
	Summing the two equations, we obtain $B_{33}=0$, i.e. $M$ is Einstein
	by Proposition \ref{suffcond}. Therefore, by the same system,
	we obtain that $S=12/t^2$.
	
%	Finally, suppose that $Z\in\AK$. By our previous considerations,
%	this means that $d\omega_+=0$; therefore,
%	the vanishing of all the components in \eqref{kahtwistplus}
%	implies the following equations:
%	\[
%	\begin{cases}
%		B_{12}=B_{13}=B_{22}=B_{23}=B_{32}=B_{33}=0, \\
%		A_{12}=A_{13}=A_{23}=0, \\
%		A_{22}=A_{33}=\dfrac{1}{t^2}.
%	\end{cases}
%	\]
%	Thus, $M$ is Einstein, self-dual with $S=12/t^2$.
%	
%\smallskip
%
%$(2)$ Let us consider the semi-K\"{a}hler equation
%	\[
%	\sum_{p=1}^6 J_{p,p}^t+J_p^rJ_p^sJ_{s,r}^t=2\sum_{p=1}^6J_{p,p}^t=0
%	\]
%	for every $t=1,\ldots,6$. It is easy to see that the left-hand side
%	of this equation is identically zero if $t=1,\ldots,4$. Thus, by putting
%	$t=5,6$, we obtain that the equation holds, i.e., $Z\in\SK$ if and
%	only if $A_{12}=A_{13}=0$, i.e.
%	if and only if $M$ is self-dual (equivalently, by \eqref{codifftwist},
%	it is clear that $\delta\omega=0$ if and only if $A_{12}=A_{13}=0$).
\end{proof}
Now, we prove the analogous of Theorem \ref{linselfdual} for $J_-$:
\begin{theorem} \label{linselfdual2}
Let $(M,g)$ be a Riemannian four-manifold and let $(Z,g_t,J_-)$ be its twistor space. If
$\nabla J_-$ satisfies the equation in \eqref{lincond} for every $X,Y\in\mathfrak{X}(Z)$,
then $M$ is self-dual.
\end{theorem}
\begin{proof}
The first step of the proof is identical to the one in Theorem \ref{linselfdual};
thus, if at least one of the coefficients $a_2+a_6$, $a_8-a_4$, $a_1+a_5$ or
$a_7-a_3$ is different from $0$,
we conclude that $M$ is self-dual. 
%Now, suppose that
%\[
%a_1=-a_5,\quad a_2=-a_6, \quad a_3=a_7,\quad a_4=a_8.
%\]
%If we choose $p=5,q=3,t=1$ and $p=1,q=5,t=3$ in \eqref{lincondort},
%we obtain the system
%\[
%\begin{cases}
%a_3\left(A_{22}+A_{33}-\dfrac{2}{t^2}\right)+a_4(A_{22}+A_{33})=0,\\
%a_3(A_{22}+A_{33})-a_4\left(A_{22}+A_{33}-\dfrac{2}{t^2}\right)=0.
%\end{cases}
%\]
%If $A_{22}+A_{33}=\frac{2}{t^2}$ or $A_{22}=-A_{33}$, we obtain that
%$a_3=a_4=0$. If $A_{22}+A_{33}\neq\frac{2}{t^2},0$, we have again
%that $a_3=a_4=0$. Note that, exchanging the values of $p$ and $q$ in
%the previous argument, we obtain the system
%\[
%\begin{cases}
%a_1\left(A_{22}+A_{33}-\dfrac{2}{t^2}\right)+a_2(A_{22}+A_{33})=0,\\
%a_1(A_{22}+A_{33})-a_2\left(A_{22}+A_{33}-\dfrac{2}{t^2}\right)=0,
%\end{cases}
%\]
%which implies that $a_1=a_2=0$. Since not all the coefficients
%$a_i$ can be equal to zero, there is a contradiction and the
%claim is proved.
Now, suppose that
\[
a_1=-a_5,\quad a_2=-a_6, \quad a_3=a_7,\quad a_4=a_8:
\]
by choosing
$t=1,q=3,p=5$ and $t=3,q=5,p=1$, we have
\[
\begin{cases}
	a_3\pa{A_{22}+A_{33}-\frac{2}{t^2}}+a_4\pa{A_{22}+A_{33}}&=0\\
	a_4\pa{A_{22}+A_{33}-\frac{2}{t^2}}+a_3\pa{A_{22}+A_{33}}&=0.
\end{cases}
\]
If $a_3\neq a_4$, it is immediate to prove that the previous
system admits no solution: indeed, if, in addition, $a_3\neq -a_4$,
the only possible solution would be 
$A_{22}+A_{33}+\frac{2}{t^2}=A_{22}+A_{33}=0$, which is impossible,
while, if $a_3=-a_4$, we would get $\frac{2}{t^2}=0$.
Hence, let $a_3=a_4$: in this case, the system reduces to 
the equation
\[
a_3\pa{A_{22}+A_{33}-\dfrac{1}{t^2}}=0,
\] 
which implies that, if $a_3\neq 0$, $(M,g)$ is a self-dual manifold
with $S=6/t^2$. Moreover, choosing $t=1,q=3,p=6$, we conclude that 
$a_1=a_2$: therefore, the relations between the $a_i$'s are
\[
a_1=a_2=-a_5=-a_6, \quad a_3=a_4=a_7=a_8.
\]
Finally, if $a_3=0$, we can choose again $t=1,q=3,p=6$ and 
$t=3,q=6,p=1$ in order to find
\[
\begin{cases}
	a_1\pa{A_{22}+A_{33}-\frac{2}{t^2}}+a_2\pa{A_{22}+A_{33}}&=0\\
	a_2\pa{A_{22}+A_{33}-\frac{2}{t^2}}+a_1\pa{A_{22}+A_{33}}&=0
\end{cases};
\]
as before, we must have $a_1=a_2$ and, since $a_1\neq 0$ by hypothesis, we 
obtain again that $(M,g)$ is self-dual with $S=6/t^2$.
\end{proof}

\section{Quadratic conditions: proofs of Theorems \ref{TH:LocalQuadratic}, \ref{TH:LocalQuadratic2}}\label{secQuadr}

In this section we present some new results
about the twistor space associated to an Einstein four-dimensional
manifold whose metric is not necessarily self-dual; in
particular, we partially generalize the characterization
Theorem \ref{muskar1} and we introduce a new necessary
condition for a manifold to be Einstein, which leads to a
characterization of Einstein, non-self-dual manifolds.

As before, let $(M,g)$ be a connected, oriented Riemannian
manifold of dimension $4$ and $(Z,g_t,J_+)$ be its twistor space,
with $J_+$ the Atiyah-Hitchin-Singer almost complex structure
defined in \eqref{twistalmcomp}. 
Since, in this section, we will only consider $J_+$, for the sake of 
simplicity, we will write $J_{q,t}^p$
instead of $(J_+)_{q,t}^p$ and 
$N_{pq}^t$ instead of $(N_+)_{pq}^t$. 
We have the following (see Theorem \ref{TH:LocalQuadratic} in the 
introduction):
\begin{theorem} \label{einstquadr}
$(M,g)$ is Einstein if and only if, for every orthonormal frame in
$O(M)_-$ (equivalently, for every negatively oriented orthonormal coframe),
\begin{equation} \label{quadrcond}
\sum_{t=1}^6(J_{p,q}^t+J_{q,p}^t)(J_{p,p}^t-J_{q,q}^t)=0, \quad\forall p,q=1,\ldots,6.
\end{equation}
\end{theorem}
\begin{proof}
First, suppose that $M$ is Einstein. Then, a direct computation over
the components of $\nabla J$ shows that \eqref{quadrcond} holds;
indeed, for instance, by Theorem \ref{cond4mani}
\[
\sum_{t=1}^6(J_{1,3}^t+J_{3,1}^t)(J_{1,1}^t-J_{3,3}^t)=B_{32}B_{12}+B_{13}B_{33}=0.
\]
Conversely, suppose that \eqref{quadrcond} holds. By choosing
$p=1,q=3$ and $p=1,q=4$, we obtain the following system
\[
\begin{cases}
	B_{12}B_{22}+B_{13}B_{23}=0;\\
	B_{12}B_{32}+B_{13}B_{33}=0.
\end{cases}
\]
By hypothesis, the two equations must hold on all $O(M)_-$. Let us
choose $e\in O(M)_-$ and suppose that $B_{ij}=0$ for some $i,j$; we
want to prove that $B=0$. For instance, let $B_{11}=0$ (the
other cases can be shown analogously). Let us choose a
smooth change of frames $a\in SO(4)$ such that
\[
a_+=a_-=
\pmat{
	0 & 1 & 0\\
	-1 & 0 & 0\\
	0 & 0 & 1};
\]
by \eqref{abctrans}, we obtain that the matrix $\tilde{B}$ associated
to the frame $\tilde{e}=ea$ has the form
\[
\tilde{B}=
\pmat{
	\tilde{B_{11}} & \tilde{B_{12}} & \tilde{B_{13}}\\
	\tilde{B_{21}} & \tilde{B_{22}} & \tilde{B_{23}}\\
	\tilde{B_{31}} & \tilde{B_{32}} & \tilde{B_{33}}
}=
\pmat{
	B_{22} & - B_{21} & -B_{23}\\
	-B_{12} & 0 & B_{13}\\
	-B_{32} & B_{31} & B_{33}.
}
\]
By hypothesis, we have that
\[
0=\tilde{B}_{12}\tilde{B}_{22}+\tilde{B}_{13}\tilde{B}_{23}=-B_{23}B_{13},
\]
that is, $B_{23}=0$ or $B_{13}=0$. In both cases, with similar computations,
it can be shown that all the other $B_{ij}$ are zero, i.e. $B=0$, which means
that $M$ is Einstein by \eqref{abctrans}. In particular, if one of the
$B_{ij}$ in the system is $0$, then $M$ is Einstein.

Let us now suppose $B_{12},B_{13},B_{22},B_{23},B_{32},B_{33}\neq 0$.
%Subtracting the two equations, we obtain
%\[
%B_{12}(B_{22}-B_{32})+B_{13}(B_{23}-B_{33})=0;
%\]
%since $B_{12},B_{13}\neq 0$, we need to consider two cases. First, suppose
%that $B_{22}=B_{32}$; then, since $B_{13}\neq0$, $B_{23}=B_{33}$.
%This means that the matrix $B$, on $O(M)_-$, has the form
%\[
%B=
%\begin{pmatrix}
%B_{11} & B_{12} & B_{13}\\
%B_{21} & B_{22} & B_{23}\\
%B_{31} & B_{22} & B_{23}\\
%\end{pmatrix};
%\]
%by the transformation law in \eqref{abctrans}, choosing
%\[
%a_+=I_3, \qquad
%a_-=
%\begin{pmatrix}
%1 & 0 & 0\\
%0 & 0 & 1\\
%0 & -1 & 0
%\end{pmatrix},
%\]
%we obtain that the transformed matrix $\tilde{B}$ has the form
%\[
%\tilde{B}=
%\begin{pmatrix}
%\tilde{B}_{11} & \tilde{B}_{12} & \tilde{B}_{13}\\
%\tilde{B}_{21} & \tilde{B}_{22} & \tilde{B}_{23}\\
%\tilde{B}_{31} & \tilde{B}_{22} & \tilde{B}_{23}
%\end{pmatrix}
%=
%\begin{pmatrix}
%B_{11} & B_{12} & B_{13}\\
%-B_{31} & -B_{22} & -B_{23}\\
%B_{21} & B_{22} & B_{23}
%\end{pmatrix},
%\]
%that is, $B_{22}=B_{32}=B_{23}=B_{33}=0$, which is a contradiction.
%If we suppose $B_{23}=B_{33}$, we have the same conclusion.
%Thus, let us suppose that $B_{22}\neq B_{32}$ and $B_{23}\neq B_{33}$.
By the previous system of equations, we obtain that
\[
B_{12}=-\dfrac{B_{13}B_{23}}{B_{22}}=-\dfrac{B_{13}B_{33}}{B_{32}},
\]
which implies that
\[
B_{23}B_{32}=B_{22}B_{33}.
\]
Choosing the matrices
\[
a_+=I_3, \qquad
a_-=
\begin{pmatrix}
	0 & 0 & 1\\
	0 & 1 & 0\\
	-1 & 0 & 0
\end{pmatrix}
\]
in \eqref{abctrans}, we deduce
\[
B_{13}B_{22}=B_{12}B_{23}=-\dfrac{B_{13}B_{23}}{B_{22}}B_{23},
\]
that is,
\[
B_{22}=-\dfrac{B_{23}^2}{B_{22}}.
\]
But this implies that $B_{22}^2=-B_{23}^2$, i.e. $B_{22}=B_{23}=0$,
which is a contradiction. This means that at least one of
the terms in the system above has to be equal to $0$; therefore,
by the previous considerations, $M$ is Einstein.
\end{proof}
We highlight that the fact that Theorem \ref{einstquadr} is only true
\emph{locally}: this means that, if $M$ is Einstein, the condition
\eqref{quadrcond} holds only for orthonormal frames. Indeed,
if $X,Y\in\mathfrak{X}(Z)$, the global version of the equation
\eqref{quadrcond}, namely
\begin{equation*}
  \pair{\pa{\nabla_XJ_+}\pa{Y}+\pa{\nabla_YJ_+}\pa{X}, \pa{\nabla_XJ_+}\pa{X}-\pa{\nabla_YJ_+}\pa{Y}}=0,
  \end{equation*}
is not satisfied, in general, if the norm of $X$ or $Y$ is different
from $1$ (for instance, it is sufficient to consider $Y=2X$). However,
it is important to underline that, in order to find characterizations
of the Einstein manifolds \textit{via} polynomial conditions on $\nabla J_+$,
we have to investigate equations of order higher than $1$, since, by
Theorem \ref{linselfdual}, every linear condition on $\nabla J_+$ implies
that $M$ is self-dual.

Now we show the following (see Theorem \ref{TH:LocalQuadratic2} in the introduction):
\begin{theorem}\label{einstquadr2}
	Let $(M,g)$ be an oriented Riemannian Einstein four-manifold. Then, for every orthonormal frame in
	$O(M)_-$,
	\begin{equation} \label{nijeinst}
		\pa{J^t_{q, p}+J^t_{p, q}}N_{pq}^t=0 \qquad \text{(no sum over $p$,$q$,$t$)}.
	\end{equation}
	Conversely, if 
	the equation \eqref{nijeinst} holds on $O(M)_-$, then, for every point
	$p\in M$ such that $\abs{\weyl^-}\neq 0$ at $p$, $B=0$. In particular, if 
	$\abs{\weyl^-}\neq 0$ on $M$, $(M,g)$ is an Einstein manifold.
\end{theorem}
\begin{proof}
	If $(M,g)$ is Einstein, we obtain that
	\begin{align} \label{nearnij}
		(J_{3,1}^5+J_{1,3}^5)N_{13}^5&=-(J_{4,2}^5+J_{2,4}^5)N_{24}^5=t^2B_{32}(A_{22}-A_{33});\\
		(J_{4,1}^5+J_{1,4}^5)N_{14}^5&=-(J_{3,2}^5+J_{2,3}^5)N_{23}^5=-2t^2B_{22}A_{23}; \notag \\
		(J_{3,1}^6+J_{1,3}^6)N_{13}^6&=-(J_{4,2}^6+J_{2,4}^6)N_{24}^6=2t^2B_{33}A_{23}; \notag \\
		(J_{4,1}^6+J_{1,4}^6)N_{14}^6&=-(J_{3,2}^6+J_{2,3}^6)N_{23}^6=t^2B_{23}(A_{22}-A_{33}); \notag
	\end{align}
	since $B=0$ by hypothesis, \eqref{nijeinst} holds. Conversely, 
	let $p\in M$ such that $\abs{\weyl^-}\neq 0$ at $p$: this 
	implies that $A$ is not a scalar matrix.
	Since the matrix $A$ is symmetric, there
	exists $e\in O(M)_-$ such that $A$ is diagonal, i.e.
	\[
	A=
	\pmat{
		x & 0 & 0\\
		0 & y & 0\\
		0 & 0 & z
	};
	\]
	since $A$ is not scalar, we can assume that $y\neq z$. Let
	\[
	a_+=
	\pmat{
		1 & 0 & 0\\
		0 & \frac{1}{\sqrt{2}} & -\frac{1}{\sqrt{2}}\\
		0 & \frac{1}{\sqrt{2}} & \frac{1}{\sqrt{2}}
	}\in SO(3);
	\]
	by \eqref{abctrans}, we have that, with respect to the
	transformed frame $\tilde{e}\in O(M)_-$,
	\[
	A=
	\pmat{
		x & 0 & 0\\
		0 & \frac{1}{2}(y+z) & \frac{1}{2}(z-y)\\
		0 & \frac{1}{2}(z-y) & \frac{1}{2}(y+z)
	}.
	\]
	By hypothesis, \eqref{nijeinst} holds on $O(M)_-$, which implies
	that $B_{22}=B_{33}=0$ with respect to $\tilde{e}$, by
	\eqref{nearnij} and the fact that $y\neq z$. Putting
	\[
	a_+=
	\pmat{
		1 & 0 & 0\\
		0 & 0 & -1\\
		0 & 1 & 0
	},
	\]
	by the transformation laws \eqref{abctrans}, we obtain
	\[
	\tilde{A}=
	\pmat{
		x & 0 & 0\\
		0 & \frac{1}{2}(y+z) & \frac{1}{2}(y-z)\\
		0 & \frac{1}{2}(y-z) & \frac{1}{2}(y+z)
	}, \qquad
	\tilde{B}=
	\pmat{
		B_{11} & B_{13} & -B_{12}\\
		B_{21} & B_{23} & 0\\
		B_{31} & 0 &-B_{32}
	};
	\]
	by \eqref{nijeinst}, $B_{23}=B_{32}=0$ (note that
	$\tilde{A}_{23}\neq 0$). By \eqref{abctrans}, if
	$a\in SO(4)$ is a change of frames such that $\mu(a)=(a_+,a_-)$,
	$A$ is invariant under the action of $a_-$: therefore, by similar
	computations on $B$ with $a_-=I_3$, it is easy to show that $B_{12}=B_{13}=0$.
	Now, let us consider two cases:
	\begin{enumerate}
		\item {\bf Case } $\mathbf{x\neq\frac{1}{2}(y+z)}$. Putting
		\[
		a_+=
		\pmat{
			0 & 0 & 1\\
			0 & 1 & 0\\
			-1 & 0 & 0
		},
		\]
		we have
		\[
		\tilde{A}=
		\pmat{
			\frac{1}{2}(y+z) & \frac{1}{2}(y-z) & 0\\
			\frac{1}{2}(y-z) & \frac{1}{2}(y+z) & 0\\
			0 & 0 & x
		}, \qquad
		\tilde{B}=
		\pmat{
			0 & 0 & B_{11}\\
			0 & 0 & B_{21}\\
			0 & 0 & B_{31}
		};
		\]
		since $\tilde{A}_{22}\neq\tilde{A}_{33}$, by \eqref{nijeinst}
		we obtain $B_{21}=0$. Again, since $A$ is invariant under the action
		of $a_-$, by analogous computations $B_{11}=0$. Finally,
		putting
		\[
		a_+=
		\pmat{
			0 & -1 & 0\\
			1 & 0 & 0\\
			0 & 0 & 1
		},
		\]
		we conclude that $B_{31}=0$, that is, $M$ is Einstein.
		\item {\bf Case } $\mathbf{x=\frac{1}{2}(y+z)}$. In this case,
		$A$ and $B$ have the form
		\[
		A=
		\pmat{
			x & 0 & 0\\
			0 & x & x-z\\
			0 & x-z & x
		}, \qquad
		B=
		\pmat{
			B_{11} & 0 & 0\\
			B_{21} & 0 & 0\\
			B_{31} & 0 & 0
		}
		\]
		(note that $x\neq z$, otherwise $\weyl^-=0$ at $p$). Choosing
		\[
		a_+=
		\pmat{
			\frac{1}{\sqrt{2}} & -\frac{1}{\sqrt{2}} & 0\\
			\frac{1}{\sqrt{2}} & \frac{1}{\sqrt{2}} & 0\\
			0 & 0 & 1
		},
		\]
		we obtain
		\[
		\tilde{A}=
		\pmat{
			x & 0 & \frac{1}{\sqrt{2}}(x-z)\\
			0 & x & \frac{1}{\sqrt{2}}(x-z)\\
			\frac{1}{\sqrt{2}}(x-z) & \frac{1}{\sqrt{2}}(x-z) & x
		}, \qquad
		\tilde{B}=
		\pmat{
			\frac{1}{\sqrt{2}}B_{11} & -\frac{1}{\sqrt{2}}B_{11} & 0\\
			\frac{1}{\sqrt{2}}B_{21} & -\frac{1}{\sqrt{2}}B_{21} & 0\\
			\frac{1}{\sqrt{2}}B_{31} & -\frac{1}{\sqrt{2}}B_{31} & 0
		};
		\]
		by $x\neq z$ and \eqref{nijeinst}, $B_{21}=0$. As we did earlier,
		the invariance of $A$ under the action of $a_-$ implies that
		$B_{11}=0$. Finally, choosing
		\[
		a_+=
		\pmat{
			\frac{1}{\sqrt{2}} & 0 & -\frac{1}{\sqrt{2}}\\
			0 & 1 & 0\\
			\frac{1}{\sqrt{2}} & 0 & \frac{1}{\sqrt{2}}
		},
		\] \\
		we conclude that $B_{31}=0$, i.e. $B=0$ at $p$.
		
	\end{enumerate}
\end{proof}
\begin{rem}
	As Theorems \ref{TH:LocalQuadratic} and \ref{TH:LocalQuadratic2} show,
	it is possible to obtain local, twistorial characterization
	of Einstein metrics on four-manifolds \emph{via} 
	quadratic polynomial conditions
	on the Atiyah-Hitchin-Singer almost complex structure: however, it
	seems much more difficult to find \emph{global} conditions for
	Einstein four-manifolds in terms of their twistor spaces, without
	assuming self-duality as a hypothesis. To the best of our
	knowledge, the only result in the literature which addresses
	the problem of global twistorial conditions for Einstein manifolds
	is due to Reznikov (\cite{reznik}), who showed that, under
	the further assumption that the sectional curvature is
	nowhere vanishing, the twistor space of an Einstein manifold
	is symplectic, with respect to a 2-form derived from 
	a computation involving linear connections which preserve
	the complex structure. We observe that, as it is 
	well-known, imposing self-duality
	on the underlying manifold implies much more rigid global conditions:
	for instance, it was observed by Hitchin that any Ricci-flat, self-dual
	manifold has a globally trivial twistor bundle (\cite{hitchinpoly}),
	which, in the compact case, tells that flat manifolds and
	certain quotients of K3 surfaces are the only four-manifolds
	with trivial twistor bundle, by a result in \cite{hitchincomp}
	(we point out that the authors obtained a global 
	twistorial characterization
	of Ricci-flat, self-dual manifolds in \cite{cdmtwist} in terms
	of the Eells-Salamon almost complex structure). Also, 
	the aforementioned results by Friedrich-Grunewald, Friedrich-Kurke
	and Hitchin show that global conditions on twistor spaces
	can be obtained in the Einstein, self-dual case 
	(\cite{frigru, frikur, hit}), together with
	the classification result due to Mu\v{s}karov \cite{mus}
	(see also the references cited in the Introduction); 
	another example could be the necessary and sufficient condition obtained by 
	Davidov, Grantcharov and Mu\v{s}karov, who showed that,
	for every Einstein, self-dual manifold, the Nijenhuis tensor
	of the Eells-Salamon almost complex structure is covariantly
	constant with respect to the Chern connection (\cite{davchern}). 
	We point out that (anti-)self-duality itself may lead to 
	strong global conclusions in terms of twistor spaces: for instance,
	LeBrun (\cite{lebrunexplicit}) 
	showed that, on the connected sum $M=m\ol{\mathbb{CP}^2}$,
	where $\ol{\mathbb{CP}^2}$ is the complex projective space
	with the opposite of the standard orientation and $m\geq 3$,
	there exists an anti-self-dual metric such that the 
	twistor space of $M$ is \emph{Moishezon}, i.e. 
	bimeromorphic to a projective variety, while Campana 
	managed to show a converse of this result (\cite{campana}). 
\end{rem}
\section{Data availability statements}
Data sharing not applicable to this article as no datasets were generated or analysed during the current study.

\section{Acknowledgements}
The authors would like to express their gratitude
to the anonymous referees for the useful suggestions and improvements. 
The authors are members of the Gruppo Nazionale per le
Strutture Algebriche, Geometriche e loro Applicazioni (GNSAGA) of INdAM
(Istituto Nazionale di Alta Matematica).

\bibliographystyle{abbrv}
\bibliography{Bibliotwist}

\end{document}